\documentclass[12pt]{amsart}

\usepackage[all]{xy}
\usepackage{amssymb,amsmath,amsthm,amscd}
\usepackage{hyperref}
\usepackage{tikz}
\usetikzlibrary{arrows.meta}
\hypersetup{colorlinks=true,linkcolor=blue,citecolor=magenta}
\usepackage{enumitem}
\usepackage{verbatim}
\usepackage{graphicx}
\usepackage{float}
\usepackage{color}
\usepackage{multirow}
\usepackage{caption}
\usepackage[mathscr]{eucal}
\usepackage{bbm}
\usepackage[textheight=8.65in, textwidth=6.8in]{geometry}
\numberwithin{equation}{section}

\newtheorem{theorem}{Theorem}[section]

\newtheorem{lem}[theorem]{Lemma}

\newtheorem{corollary}[theorem]{Corollary}

\theoremstyle{definition}
\newtheorem{definition}[theorem]{Definition}

\theoremstyle{remark}

\newtheorem*{thm*}{Theorem}
\newtheorem*{conj*}{Conjecture}

\newcommand{\Z}{\mathbb{Z}}

\newcommand{\PP}{\mathbb{P}}
\newcommand{\SL}{\operatorname{SL}}
\newcommand{\C}{\mathbb{C}}

\begin{document}

\title[Products of Eisenstein Series with Multiplicative Power Series]{Products of Eisenstein Series with Multiplicative Power Series}
\author{Boyuan Xiong}
\address{Department of Mathematics, Indiana University, Bloomington, IN 47405, USA}
\email{boyxiong@iu.edu}

\begin{abstract}
We say a power series $a_0+a_1q+a_2q^2+\cdots$ is \emph{multiplicative} if $n\mapsto a_n/a_1$ for positive integers $n$ is a multiplicative function.  Given the Eisenstein series $E_{2k}(q)$, we consider formal multiplicative power series $g(q)$ such that the product $E_{2k}(q)g(q)$ is also multiplicative. For fixed $k$, this requirement leads to an infinite system of polynomial equations in the coefficients of $g(q)$. The initial coefficients can be analyzed using elimination theory. Using the theory of modular forms, we prove that each solution for the initial coefficients of $g(q)$ leads to one and only one solution for the whole power series, which is always a quasimodular form. In this way, we determine all solutions of the system for $k \le 20$.

For general $k$, we can regard the system of polynomial equations as living over a symbolic ring. Although this system is beyond the reach of computer algebra packages, we can use a specialization argument to prove it is generically inconsistent. This is delicate because resultants commute with specialization only when the leading coefficients do not specialize to $0$. Using a Newton polygon argument, we are able to compute the relevant degrees and justify the claim that for $k$ sufficiently large, there are no solutions.

These results support the conjecture that $E_{2k}(q)g(q)$ can be multiplicative only for $k = 2, 3, 4, 5, 7$.
\end{abstract}

\maketitle

\section{Introduction}

If $f(q)$ and $g(q)$ are $q$-expansions of normalized Hecke eigenforms, then the coefficients $a_n$ and $b_n$ of $f$ and $q$ respectively are both multiplicative sequences in the sense that $a_{mn} = a_m a_n$ for $(m,n)=1$ and likewise for $b$.
There is no obvious relation between the Hecke operators at different weights and the product structure on the ring of modular forms, so in general one would not expect the product of two Hecke eigenforms to again be a Hecke eigenform.
However, it sometimes happens that $f(q) g(q)$ has multiplicative coefficients, often because it lies in a $1$-dimensional space on which the relevant Hecke operators act.

For example, let $M_k$ (resp. $S_k$) denote the space of modular forms (resp. cusp forms) for $\SL_2(\Z)$ of weight $k$. Then, $\dim M_4 = \dim S_{12} = \dim S_{16} = 1$, so we must have the identity
$$E_4(q) \Delta_{12}(q) = \Delta_{16}(q),$$
where
\begin{equation}
\label{one}
E_{2k}(q) := 1 - \frac{4k}{B_{2k}}\sum_{n=1}^\infty \sigma_{2k-1}(n) q^n
\end{equation}
is a (non-normalized) Hecke eigenform of level $1$ and weight $2k$ 
and $\Delta_m(q)$ is the unique normalized cusp form of level $1$ and weight $m$ for $m\in \{12,16\}$.

Matthew Johnson \cite{Johnson} gave a complete classification of pairs $(f,g)$ of Hecke eigenforms on $\Gamma_1(N)$
whose product is again an eigenform, following earlier results of Ghate \cite{Ghate}, Duke \cite{Duke}, and Emmons\cite{Emmons}.
In each case, at least one of the two factors is an Eisenstein series, though not necessarily exactly of the form \eqref{one}; there may be a character or an omitted Euler factor.

In this paper, we generalize this as follows. We ask for products $f(q) g(q)$, where $f(q)$ is of the form \eqref{one}, $g(q)$ is \emph{any} power series $q + \sum_{n=2}^\infty b_n q^n$, where the $b_n$ are multiplicative, and the coefficients $c_n$ in  
$$f(q) g(q) = q + \sum_{n=2}^\infty c_n q^n$$
are again multiplicative. A basic question is whether for all such products $g(q)$ is the $q$-expansion of a modular form. 

Regarding the coefficients $b_n$, where $n$ is a prime power, as degrees of freedom in choosing $g(q)$, the 
relations $c_{mn} = c_m c_n$ for $(m,n)=1$, are constraints. Since most integers are not prime powers, this is in some sense an overdetermined system of polynomial equations, and we therefore expect solutions to be in some sense rare and special.

This problem is similar in spirit to the question considered by Larsen \cite{Larsen} about  power series $f(q)$ with multiplicative coefficients such that $f(q)^2$ also has multiplicative coefficients. Larsen found solutions which are essentially Eisenstein series and others which are rational functions. He did not prove that his list of solutions was complete but did show that the full solution set can be identified with the points on a finite dimensional variety (not necessarily irreducible).

In this paper, we give an exhaustive list of solutions for all $k\le 20$. With one exception, the $g(q)$ are indeed modular forms; the exception is still quasimodular in the sense of Kaneko-Zagier \cite{KZ}.
We also show that for $k$ sufficiently large, there are no solutions. From this it easily follows that the set of multiplicative $g(q)$ for which there exists $k$ such that $E_{2k}(q)g(q)$ is also multiplicative can be identified with the points of a finite dimensional variety.

Our basic strategy is to analyze the system of polynomial equations in the variables $b_{p^i}$ determined by the equations of the form 
$c_{mn} = c_m c_n$. The coefficients of these equations can be expressed as polynomials in terms of $q^{2k-1}$, where $q$ ranges over the primes, and $\frac{B_{2k}}{4k}$. 
For fixed $k$, 
this system becomes overdetermined  when one looks at coefficients up to $40$.
At this point, we are dealing with $19$ variables. Many of the equations are linear in some of the variables, but even after this fact has been exploited, 
the system is near the limit of what can be handled by computer algebra systems. 
It is not difficult to show that the values of $b_n$ for $n\le 9$ determine $g(q)$, so we need only find, for each solution to our system of polynomial equations, an actual modular or quasimodular function with the specified 
initial coefficients, which can be proved to have the desired multiplicativity properties.

To rule out solutions for large $k$ is substantially more difficult. The coefficients of $f(q)$ depend on $k$ in an exponential way. Our strategy is to treat expressions of the form $p^{k-1}$ as parameters, perform the computer algebra computations, and then solve for $k$ at the end. The trouble is that this requires solving a system of $20$ equations in $19$ variables over a polynomial ring in the $13$ variables identified with $\frac{B_{2k}}{4k},2^{2k-1}, 3^{2k-1}, 5^{2k-1},\ldots, 37^{2k-1}$, which appears to be well beyond the limits of existing computer algebra systems.

What makes it nevertheless possible to prove our theorem is that a system of equations can be proved inconsistent by a specialization argument. 
Consider the $n$ polynomials $f_{1,1},\ldots,f_{1,n}$ in $n-1$ variables $x_1,\ldots,x_{n-1}$ over an integral domain $A$.
Now consider a homomorphism 
$\phi\colon A\to B$ to another integral domain. We denote by $\bar f_{1,i}$ the image of $f_{1,i}$.  Let $K$ and $L$ denote the fraction fields of $A$ and $B$ respectively. For $1\le i\le j-1 < n$, we define $f_{i+1,j}$ to be the resultant of $f_{i,i}$ and $f_{i,j}$ with respect to $x_i$.
If $(a_1,\ldots,a_{n-1})\in \bar K^{n-1}$ is a solution of all $f_{1,j}=0$, then by induction, $(a_{i+1},\ldots,a_{n-1})$ is a solution of all $f_{i+1,j}=0$.
If $f_{n,n}\in A$ is non-zero, therefore, the system has no solution in $\bar K$.

Now if the $x_i$-degrees of $f_{i,j}$ and $\bar f_{i,j}$ are the same for all $i\le j\le n$, then taking resultants commutes with applying $\phi$, so it suffices to prove that
$\phi(f_{n,n}) = \bar f_{n,n}\neq 0$, iteratively computing $\bar f_{i,j}$, by taking resultants in $B$. To make this work, however, we need to keep track of the $x_i$-degrees of the $f_{i,j}$ and make sure they always match the $x_i$-degrees of the $\bar f_{i,j}$. Computing the degrees of the $f_{i,j}$ is therefore the main task of this paper; fortunately, it is much easier than computing  the polynomials themselves. We make essential use of the Bernstein theorem \cite{Minding} \cite{Bernstein}\cite{GKZ} at the key step of reducing from three equations in two unknowns to two equations in one unknown.

The actual strategy is slightly more complicated than what is described here because it turns out that if one follows this elimination procedure for our particular polynomials, at some stage we obtain $f_{i,i} = \cdots = f_{i,n} = 0$. This happens because there is a common factor $g$ among $f_{i-1,i-1},\ldots,f_{i-1,n}$. 
By removing this factor from all the $f_{i-1,j}$, we obtain a new sequence, with which we proceed as before, but we need to consider separately the solutions where $g=0$. This leads to a new, less computationally demanding, analysis of the same kind.

\section{Solutions for small values of \texorpdfstring{$k$}{k}}

Let $\mathbb{P}$ denote the set of prime powers:
\[
\mathbb{P} = \{p^n : p \ \text{prime},\ n>0\}.
\]
Suppose
\[
f(q) g(q) = q+\sum_{n=2}^\infty c_{k,n}q^n
\]
with
\[
f(q)=E_{2k}(q),\qquad
g(q)=q+\sum_{n=2}^\infty b_n q^n.
\]
Then
\[
c_{k,n} = b_n -\frac{4k}{B_{2k}} \sum_{i = 1}^{n-1} \sigma_{2k-1}(i)\, b_{n-i}
\quad\text{for all } n \geq 2.
\]
From the relation $c_{k,mn} = c_{k,m}c_{k,n}$ we deduce
\[
\Bigg(b_m -\frac{4k}{B_{2k}} \sum_{i = 1}^{m-1} \sigma_{2k-1}(i)b_{m-i}\Bigg)
\Bigg(b_n -\frac{4k}{B_{2k}} \sum_{i = 1}^{n-1} \sigma_{2k-1}(i)b_{n-i}\Bigg)
= b_{mn} -\frac{4k}{B_{2k}} \sum_{i = 1}^{mn-1} \sigma_{2k-1}(i)b_{mn-i}
\]
whenever $(m,n)=1$. Expanding the left-hand side and using $b_{mn} = b_mb_n$ gives
\begin{equation}\label{two}
S_{k,mn} = b_mS_{k,n} + b_nS_{k,m} - \frac{4k}{B_{2k}}S_{k,m}S_{k,n},
\end{equation}
where $S_{k,n} = \sum_{i=1}^{n-1}\sigma_{2k-1}(i)b_{n-i}$.

Define
\[
E_{k,mn} = S_{k,mn} - \big(b_mS_{k,n} + b_nS_{k,m} - \tfrac{4k}{B_{2k}}S_{k,m}S_{k,n}\big)
\quad \text{for all } (m,n)=1.
\]
\footnote{The definition of $E_{k,n}$ in \eqref{two} is not unique, since $n$ may factor into coprime integers in different ways (e.g.\ $30=2\cdot 15=3\cdot 10=5\cdot 6$). Nevertheless, it suffices to take any one such relation for each $n\notin \mathbb{P}$.}

Thus our system consists of the variables $b_n$, together with the polynomials $E_{k,n}$ and $b_{mn}-b_mb_n$. We focus on the subsystem containing all $b_n$ and $E_{k,n}$ with $n\leq 40$, namely
\[
(\{b_n\}_{n\leq 40};\ \{E_{k,n}\}_{n\leq 40}\cup \{b_{mn}-b_mb_n\}_{mn\leq 40}).
\]
Our goal is to solve this system using elimination. For this we record a useful lemma.

\begin{lem}\label{lemma1}
Suppose $l\in\mathbb{P}$ and $n\notin\mathbb{P}$. If $\tfrac{n}{2}<l<n$, then $b_l$ appears linearly in $E_{k,n}$ with
\[
\operatorname{coef}_{b_l}(E_{k,n})=\sigma_{2k-1}(n-l).
\]
In particular, if $p\in\mathbb{P}$ and $p+1\notin\mathbb{P}$, then $b_p$ appears linearly in $E_{k,p+1}$ with $\operatorname{coef}_{b_p}(E_{k,p+1})=1$.
\end{lem}

\begin{proof}
Write $n=xy$ with $(x,y)=1$. Then
\[
E_{k,n}=S_{k,n}-\Big(b_xS_{k,y}+b_yS_{k,x}-\tfrac{4k}{B_{2k}}S_{k,x}S_{k,y}\Big).
\]
Since $\max\{x,y\}\leq n/2<l$, the variable $b_l$ cannot occur in $S_{k,x}$, $S_{k,y}$, or $S_{k,x}S_{k,y}$. Hence the only contribution of $b_l$ comes from $S_{k,n}$. Because $\operatorname{coef}_{b_l}(S_{k,n})=\sigma_{2k-1}(n-l)$, the claim follows.
\end{proof}

\medskip\noindent
\textbf{Step 1. Elimination of non-prime-power indices.}
Using the relations $b_{mn}=b_mb_n$, we substitute every $b_{mn}$ with $b_mb_n$, thereby eliminating all polynomials of the form $b_{mn}-b_mb_n$ and removing variables $b_n$ with $n\notin\mathbb{P}$.

\medskip\noindent
\textbf{Step 2. Linear eliminations.}
By Lemma~\ref{lemma1}, $b_p$ appears linearly in $E_{k,p+1}$ whenever $p\in\mathbb{P}$ and $p+1\notin\mathbb{P}$. Solving $E_{k,p+1}=0$ for $b_p$, we can express $b_p$ in terms of variables $\{b_n\}_{n<p}$, substitute this back into the system, and remove both $b_p$ and $E_{k,p+1}$. In addition, $b_8$ appears linearly in $E_{k,15}$, $b_{16}$ in $E_{k,21}$, and $b_{31}$ in $E_{k,34}$, permitting the removal of these pairs as well. Altogether, we eliminate
\begin{align*}
&(b_5, E_{k,6}), (b_8, E_{k,15}), (b_{11}, E_{k,12}), (b_{13}, E_{k,14}), (b_{16}, E_{k,21}), (b_{17}, E_{k,18}), (b_{19}, E_{k,20}), \\
&(b_{23}, E_{k,24}), (b_{25}, E_{k,26}), (b_{27}, E_{k,28}), (b_{29}, E_{k,30}), (b_{31}, E_{k,34}), (b_{32}, E_{k,33}), (b_{37}, E_{k,38}).
\end{align*}
The reduced system now consists of the variables and polynomials
\[
\{b_2, b_3, b_4, b_7\}, \quad
\{E_{k,22}, E_{k,35}, E_{k,36}, E_{k,39}, E_{k,40}\}.
\]
At this point the number of polynomials exceeds the number of unknowns.

\medskip\noindent
\textbf{Step 3. Elimination via resultants.}
The remaining variables no longer appear linearly in any $E_{k,n}$. To proceed, we employ resultants. Define
\begin{align*}
F_{k,35} &= \operatorname{Res}_{b_7}(E_{k,22}, E_{k,35}), &
F_{k,36} &= \operatorname{Res}_{b_7}(E_{k,22}, E_{k,36}), \\
F_{k,39} &= \operatorname{Res}_{b_7}(E_{k,22}, E_{k,39}), &
F_{k,40} &= \operatorname{Res}_{b_7}(E_{k,22}, E_{k,40}).
\end{align*}
This yields a new system
\[
(\{b_2, b_3, b_4\}; \{F_{k,35}, F_{k,36}, F_{k,39}, F_{k,40}\}),
\]
thereby eliminating $b_7$.
Next, set
\begin{align*}
G_{k,36} &= \operatorname{Res}_{b_4}(F_{k,35},F_{k,36}), \\
G_{k,39} &= \operatorname{Res}_{b_4}(F_{k,35},F_{k,39}), \\
G_{k,40} &= \operatorname{Res}_{b_4}(F_{k,35},F_{k,40}).
\end{align*}
Eliminating $b_4$ leaves the system
\[
(\{b_2,b_3\};\{G_{k,36},G_{k,39},G_{k,40}\}).
\]
At first glance, one might try to continue in the same way to eliminate $b_3$. However, computations in \texttt{SageMath} show that
\[
\operatorname{Res}_{b_3}(G_{k,i},G_{k,j})=0\quad\text{for all }i,j\in\{36,39,40\}.
\]

\medskip\noindent
\textbf{Why $\mathrm{Res}_{b_3}(G_{k,i},G_{k,j})=0$.}
For $n\in\{22,35,36,39,40\}$, writing $n=xy$ with $(x,y)=1$ shows that $\min(x,y)<7$. Hence $b_7$ appears linearly in $E_{k,n}$ for all such $n$, so we may write
\[
E_{k,n} = P_{k,n}\,b_7 + Q_{k,n},
\]
where $P_{k,n}$ and $Q_{k,n}$ are independent of $b_7$. Explicit computations in \texttt{SageMath} show that
\[
\deg_{b_4}(P_{k,22})=1,\qquad
\deg_{b_4}(P_{k,n})=2\ (n\in\{35,36,39,40\}),\qquad
\deg_{b_4}(Q_{k,n})=2\ (n\in\{22,35,36,39,40\}).
\]

Accordingly, $E_{k,22}$ takes the form
\[
A = a\,b_7 + b\,b_4^2 + c\,b_4 + d,
\]
while for $n\in\{35,36,39,40\}$ each $E_{k,n}$ can be written uniformly as
\[
B = (a'b_4+e')\,b_7 + (b'b_4^2+c'b_4+d'),
\]
with coefficients independent of $b_4,b_7$.
When comparing two different polynomials of form $B$, we denote them by
\[
B_1 = (a_1'b_4+e_1')\,b_7 + (b_1'b_4^2+c_1'b_4+d_1'),\qquad
B_2 = (a_2'b_4+e_2')\,b_7 + (b_2'b_4^2+c_2'b_4+d_2').
\]

Then any $G_{k,n}$ takes the form
\[
G \;=\; \mathrm{Res}_{b_4}\!\Big(\mathrm{Res}_{b_7}(A,B_1),\ \mathrm{Res}_{b_7}(A,B_2)\Big).
\]
After computing and factoring this expression in \texttt{SageMath}, we find that $a^2$ is a factor of $G$. Therefore, $(P_{k,22})^2$ divides each $G_{k,n}$ for $n\in\{36,39,40\}$.\footnote{This can also be seen from a standard fact: Suppose $A$ is a UFD and $\pi$ is a prime in $A$. Let $f,g \in A[x]$, and let $\bar{f},\bar{g}$ be their images in $(A/(\pi))[x]$. If $d = \deg(\gcd(\bar{f}, \bar{g}))$, then $\pi^d$ divides $\operatorname{Res}_{x}(f,g)$. In our case, take $A=\C[b_2,b_3]$, $\pi = P_{k,22}$, and $F_{k,n} = P_{k,22}Q_{k,n}-Q_{k,22}P_{k,n}\in A[b_4]$ for $n\in\{35,36,39,40\}$. Since $\deg_{b_4}(Q_{k,22})=2$ and $\overline{Q_{k,22}}$ divides each $\overline{F_{k,n}}$ in $A/(\pi)[b_4]$, we have $d\geq 2$. Therefore, $(P_{k,22})^2$ must divide each $G_{k,n}$.}
Finally, since $\deg_{b_3}(P_{k,22})>0$, we have
\[
\mathrm{Res}_{b_3}(G_{k,i},G_{k,j})=0
\quad\text{for all } i,j\in\{36,39,40\}.
\]

\medskip

Accordingly, we define
\begin{align*}
H_{k,36} &= \operatorname{Res}_{b_3}\!\Big(\tfrac{G_{k,39}}{(P_{k,22})^2},\;\tfrac{G_{k,36}}{(P_{k,22})^2}\Big),\\
H_{k,40} &= \operatorname{Res}_{b_3}\!\Big(\tfrac{G_{k,39}}{(P_{k,22})^2},\;\tfrac{G_{k,40}}{(P_{k,22})^2}\Big).
\end{align*}
We then split the problem into two systems:
\[
\text{sys}_{k,1} : (\{b_2\};\ \{H_{k,36},H_{k,40}\}),\qquad
\text{sys}_{k,2} : (\{b_2,b_3,b_4,b_7\};\ \{P_{k,22},Q_{k,22},E_{k,35},E_{k,36},E_{k,39},E_{k,40}\}).
\]

\medskip\noindent
\textbf{Step 4. Specialization to $k=2$.}
For $\text{sys}_{2,1}$, $H_{2,36}$ and $H_{2,40}$ are polynomials in the single variable $b_2$. Computing their $\gcd$ gives
\[
b_2 = -528,\ -288,\ -24,\ -8,\ 18,\ 216.
\]

For $\text{sys}_{2,2}$, define
\begin{align*}
R_{2,22} &= \operatorname{Res}_{b_3}(P_{2,22},Q_{2,22}), \\
R_{2,35} &= \operatorname{Res}_{b_3}(P_{2,22},E_{2,35}), \\
R_{2,36} &= \operatorname{Res}_{b_3}(P_{2,22},E_{2,36}), \\
R_{2,39} &= \operatorname{Res}_{b_3}(P_{2,22},E_{2,39}).
\end{align*}
Next, set
\begin{align*}
T_{2,36} &= \operatorname{Res}_{b_7}(R_{2,35},R_{2,36}), \\
T_{2,39} &= \operatorname{Res}_{b_7}(R_{2,35},R_{2,39}),
\end{align*}
and finally
\begin{align*}
U_{2,36} &= \operatorname{Res}_{b_4}(R_{2,22},T_{2,36}), \\
U_{2,39} &= \operatorname{Res}_{b_4}(R_{2,22},T_{2,39}).
\end{align*}
\texttt{SageMath} computations show that $\gcd(U_{2,36},U_{2,39})=1$, so $\text{sys}_{2,2}$ admits no solutions.

\begin{theorem}\label{thm1}
The solutions in the case $k=2$ for $g(q)$ are the following:
\begin{align*}
\Delta_{12}(q) &= q - 24q^2 + 252q^3 - 1472q^4 + 4830q^5 + \cdots,\\
\Delta_{16}(q) &= q + 216q^2 - 3348q^3 + 13888q^4 + 52110q^5 + \cdots,\\
\Delta_{18}(q) &= q - 528q^2 - 4284q^3 + 147712q^4 - 1025850q^5 + \cdots,\\
\Delta_{22}(q) &= q - 288q^2 - 128844q^3 - 2014208q^4 + 21640950q^5 + \cdots,\\
\varphi_8(q) = \eta(z)^8\eta(2z)^8  &= q - 8q^2 + 12q^3 + 64q^4 - 210q^5 + \cdots,\\
\tfrac{1}{480\pi i}\cdot\tfrac{dE_4}{dz}(q) &= q + 18q^{2} + 84q^{3} + 292q^{4} + 630q^{5} + \cdots,
\end{align*}
where $\eta(z) = q^{\frac{1}{24}}\prod_{n=1}^{\infty}(1-q^n)$, and $\varphi_8(q)$ is the normalized cusp form of level 2 and weight 8, denoted (2.8.a.a) in the online modular form database LMFDB \cite{LMFDB}.\end{theorem}

Before proving the theorem, we require the following lemma.

\begin{lem}\label{lemma2}
If $g(q) = q + \sum_{n=2}^\infty b_nq^n$ is such that both $g(q)$ and $E_{2k}(q)g(q)$ are multiplicative, then $g(q)$ is uniquely determined by its first $8$ coefficients.
\end{lem}

\begin{proof}
We proceed by induction on $l$.

\noindent\textbf{Base step.} For $l=9$, the coefficient $b_{9}$ appears linearly in $E_{k,10}$, and all other coefficients involved have index at most $9$. Thus $b_{9}$ can be expressed in terms of $\{b_n\}_{n\leq 8}$.

\noindent\textbf{Inductive step.} Assume $b_m$ is determined by $\{b_n\}_{n\leq 8}$ for all $m<l$. We distinguish three cases.

\smallskip
\emph{Case 1: $l\notin\mathbb{P}$.} Then $l=xy$ with $(x,y)=1$, so $b_l=b_xb_y$ reduces to smaller indices.

\emph{Case 2: $l\in\mathbb{P}$ but $l+1\notin\mathbb{P}$.} By Lemma~\ref{lemma1}, $b_l$ appears linearly in $E_{k,l+1}$, so $b_l$ is determined by $\{b_n\}_{n<l}$.

\emph{Case 3: $l,l+1\in\mathbb{P}$.} Then either $l$ is a Mersenne prime or $l+1$ is a Fermat prime \cite[Lemma~4.3]{Larsen}.

\quad($l$ Mersenne).
Then $3\,|\,l+2$, and $l+3$ is an even number between two consecutive powers of $2$, so $l+2,l+3\notin\mathbb{P}$. By Lemma~\ref{lemma1}, both $b_l$ and $b_{l+1}$ appear linearly in $E_{k,l+2}$ and $E_{k,l+3}$, with coefficient matrix
\[
\begin{bmatrix}
\sigma_{2k-1}(2) & \sigma_{2k-1}(3) \\
1 & \sigma_{2k-1}(2)
\end{bmatrix}
\]
whose determinant $2^{2k}+2^{4k-2}-3^{2k-1}$ is nonzero. Thus $b_l,b_{l+1}$ are determined.

\quad($l+1$ Fermat).
Then $l+2,l+4$ are even numbers between two consecutive powers of $2$, and $l+5$ is divisible by $3$ and $\equiv 3 \pmod 8$. If $l+5\in\PP$, then $l+5$ would be a perfect square between $l$ and $(\sqrt l+1)^2$, which is impossible. Thus $l+2,l+4,l+5\notin\PP$. If $l+3\notin\PP$, then $E_{k,l+2},E_{k,l+3}$ suffice to solve for $b_l,b_{l+1}$. If $l+3\in\PP$, then $b_l,b_{l+1},b_{l+3}$ appear linearly in $E_{k,l+2},E_{k,l+4},E_{k,l+5}$, with coefficient matrix
\[
\begin{bmatrix}
\sigma_{2k-1}(2) & \sigma_{2k-1}(4) & \sigma_{2k-1}(5)\\
1 & \sigma_{2k-1}(3)& \sigma_{2k-1}(4)\\
0 & 1 & \sigma_{2k-1}(2)
\end{bmatrix}
\]
whose determinant is 
\[
\begin{aligned}
    &\sigma_{2k-1}(2)(\sigma_{2k-1}(2)\sigma_{2k-1}(3) - 2\sigma_{2k-1}(4)) + \sigma_{2k-1}(5)\\
    > &\sigma_{2k-1}(2)(6^{2k-1}-2\cdot4^{2k-1})\\
    > &0
\end{aligned}
\] for all $k\geq 2$. Hence $b_l,b_{l+1},b_{l+3}$ are determined.
\end{proof}

\begin{proof}[Proof of Theorem~\ref{thm1}]
We first verify that all six series listed above are solutions. Since
\[
\dim M_4 = \dim S_{12} = \dim S_{16} = \dim S_{18} = \dim S_{20} = \dim S_{22} = \dim S_{26} = 1,
\]
it follows that
\begin{align*}
E_4(q)\Delta_{12}(q) &= \Delta_{16}(q),\\
E_4(q)\Delta_{16}(q) &= \Delta_{20}(q),\\
E_4(q)\Delta_{18}(q) &= \Delta_{22}(q),\\
E_4(q)\Delta_{22}(q) &= \Delta_{26}(q).
\end{align*}
Thus the first four series are solutions. For \(\varphi_8(q)\), Since 
\(E_4\in M_4(\mathrm{SL}_2(\mathbb Z))\) and \(\varphi_8\in S_8(\Gamma_0(2))\),
we have
\[
E_4(q)\,\varphi_8(q)\in S_{12}(\Gamma_0(2)).
\]
Consider \(\Delta_{12}(q)+256\,\Delta_{12}(q^2)\), which also lies in
\(S_{12}(\Gamma_0(2))\). A direct check of the Fourier expansions shows that
the coefficients of \(q,q^2,q^3\) in \(E_4(q)\varphi_8(q)\) and in \(\Delta_{12}(q) + 256\Delta_{12}(q^2)\)
agree. The Sturm bound \footnote{The Sturm bound \cite{Sturm} is an upper bound on the least index where the coefficients of the Fourier expansions of distinct modular forms in the same space $M_k(N,\chi)$ must differ. More precisely, the Sturm bound for $M_k(N,\chi)$ is $\lfloor \frac{km}{12} \rfloor$, where $m=[\SL_2(\Z)\colon\Gamma_0(N)] = N\prod_{p|N}(1+\frac{1}{p})$.} for \(S_{12}(\Gamma_0(2))\) is
\[
\Bigl\lfloor\frac{12}{12}\,[\mathrm{SL}_2(\mathbb Z):\Gamma_0(2)]\Bigr\rfloor
=\lfloor 1\cdot 3\rfloor=3,
\]
since \([\mathrm{SL}_2(\mathbb Z):\Gamma_0(2)]=2\bigl(1+\tfrac12\bigr)=3\).
Hence the two cusp forms are equal:
\[
E_4(q)\,\varphi_8(q)=\Delta_{12}(q)+256\,\Delta_{12}(q^2).
\]
Since $\Delta_{12}(q)$ is multiplicative, so does $\Delta_{12}(q) + 256\Delta_{12}(q^2)$ and therefore $\varphi_8(q)$ is a solution.
Finally, using $E_4(q)^2=E_8(q)$ and differentiating with respect to $z$ \footnote{$q = e^{2\pi i z}.$}, we obtain
\[
2E_4(q)\,\tfrac{dE_4}{dz}(q) = 2\pi i\,qE_8'(q).
\]
As $qE_8'(q)$ is multiplicative, it follows that $\tfrac{dE_4}{dz}(q)$ is also a solution after rescaling.

Next, we show that these are the only solutions for $k=2$. Substituting
\[
b_2 = -528,\ -288,\ -24,\ -8,\ 18,\ 216
\]
into $G_{2,36}$ and $G_{2,39}$ and computing $\gcd(G_{2,36},G_{2,39})$ yields
\[
b_3 = -4284,\ -128844,\ 252,\ 12,\ 84,\ -3348.
\]
Proceeding in this way produces the following table of coefficients:

\begin{table}[H]
\centering
\caption{Initial coefficients of the six solutions for $k=2$}
\label{tab:coeffs}
\begin{tabular}{|l|r|r|r|r|r|}
\hline
 & $b_{2}$ & $b_{3}$ & $b_{5}$ & $b_{7}$ & $b_{8}$   \\ \hline
$g_1(q)$ & $-24$ & $252$ & $4830$ & $-16744$ & $84480$ \\ \hline
$g_2(q)$ & $216$ & $-3348$ & $52110$ & $2822456$ & $-4078080$ \\ \hline
$g_3(q)$ & $-528$ & $-4284$ & $-1025850$ & $3225992$ & $-8785920$ \\ \hline
$g_4(q)$ & $-288$ & $-128844$ & $21640950$ & $-768078808$ & $1184071680$ \\ \hline
$g_5(q)$ & $-8$ & $12$ & $-210$ & $1016$ & $-512$ \\ \hline
$g_6(q)$ & $18$ & $84$ & $630$ & $2408$ & $4680$ \\ \hline
\end{tabular}
\end{table}

We observe that $g_1(q),\dots,g_6(q)$ agree with the six series identified above in their first eight coefficients. By Lemma~\ref{lemma2}, the agreement on the first eight coefficients implies that the series coincide. Hence the six series in Theorem~\ref{thm1} exhaust all solutions for $k=2$.
\end{proof}

By applying the same method for $3 \leq k \leq 20$, we obtain the following identities.

\begin{theorem}\label{theorem 2}
The solutions for $f(q)\in\{E_{2k}(q)\colon 3\leq k \leq 20\}$ are given by
\begin{align*}
E_6(q)\Delta_{12}(q) &= \Delta_{18}(q),\\
E_6(q)\Delta_{16}(q) &= \Delta_{22}(q),\\
E_6(q)\Delta_{20}(q) &= \Delta_{26}(q),\\
E_8(q)\Delta_{12}(q) &= \Delta_{20}(q),\\
E_8(q)\Delta_{18}(q) &= \Delta_{26}(q),\\
E_{10}(q)\Delta_{12}(q) &= \Delta_{22}(q),\\
E_{10}(q)\Delta_{16}(q) &= \Delta_{26}(q),\\
E_{14}(q)\Delta_{12}(q) &= \Delta_{26}(q).
\end{align*}
\end{theorem}

\section{Extreme monomials and Newton polygons}

We are going to solve the equation system for $k$ large. When $k$ is large, it is hard to compute each resultant directly, but we can determine the degree of the relevant resultants in another way.

\begin{definition}[Extreme monomials of a multivariate polynomial]
Let
\[
f(x_1,\dots,x_n)
= \sum_{\alpha \in \mathbb{Z}_{\ge 0}^n} c_\alpha x^\alpha,
\qquad x^\alpha = x_1^{\alpha_1}\cdots x_n^{\alpha_n}.
\]
Say $(\beta_1,\cdots,\beta_n)>(\alpha_1,\cdots,\alpha_n)$ if $\beta_i \geq \alpha_i$ for all $i=1,2,\cdots,n$ and $\beta_i > \alpha_i$ for at least one $i$. We say $x^\alpha = x_1^{\alpha_1}\cdots x_n^{\alpha_n}$ is an \emph{extreme monomial} of $f$ if $c_\alpha \neq 0$ and whenever $\beta>\alpha$ we have $c_\beta = 0$.
\end{definition}

Suppose we have two multivariate polynomials $f,g\in\C[x_1,x_2,\cdots,x_n]$ with generic coefficients. Then, by the definition of resultants, the extreme monomials of $\operatorname{Res}_{x_n}(f,g)$ are determined by the extreme monomials of $f$ and $g$.

\begin{definition}[Newton polygon / polytope]
Let
\[
f(x_1,\dots,x_n)
= \sum_{\alpha \in \mathbb{Z}_{\ge 0}^n} c_\alpha x^\alpha,
\qquad x^\alpha = x_1^{\alpha_1}\cdots x_n^{\alpha_n},
\]
be a nonzero polynomial over a field (or ring), where only finitely many coefficients
\(c_\alpha\) are nonzero. The \emph{Newton polygon} (or more generally the
\emph{Newton polytope}) of \(f\) is the convex hull in \(\mathbb{R}^n\)
of the exponent vectors of nonzero monomials:
\[
N(f) := \operatorname{Conv}\bigl(\{\alpha \in \mathbb{Z}_{\ge0}^n : c_\alpha \neq 0\}\bigr).
\]
In the bivariate case ($n=2$), this convex hull is a planar polygon in $\mathbb{R}^2$,
and is traditionally called the \emph{Newton polygon} of $f(x,y)$.
\end{definition}

For example, let
\[
f(x,y)= 3x^4y^2 + 2x^3 + 5xy^3 + 7y^2 + x^2y + 1.
\]
Then the Newton polygon $N(f)$ is shown below.

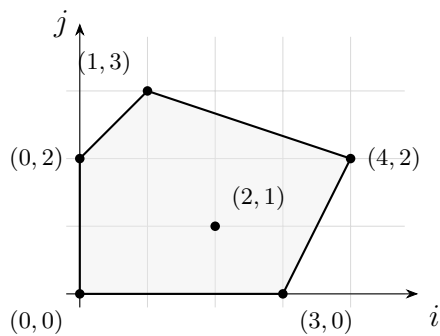
\begin{figure}[ht]
\centering
\begin{tikzpicture}[scale=0.9, >=Stealth]
 
  \draw[step=1,very thin,gray!30] (-0.2,-0.2) grid (4.8,3.8);
  \draw[->] (-0.2,0) -- (5.0,0) node[below right] {$i$};
  \draw[->] (0,-0.2) -- (0,4.0) node[left] {$j$};

  \filldraw[fill=gray!15, fill opacity=0.4, draw=black, thick]
    (0,0) -- (0,2) -- (1,3) -- (4,2) -- (3,0) -- cycle;

  \fill (0,0) circle (2pt) node[below left=2pt, font=\scriptsize] {$(0,0)$};
  \fill (0,2) circle (2pt) node[left=2pt, font=\scriptsize] {$(0,2)$};
  \fill (1,3) circle (2pt) node[above left=2pt, font=\scriptsize] {$(1,3)$};
  \fill (2,1) circle (2pt) node[above right=2pt, font=\scriptsize] {$(2,1)$};
  \fill (3,0) circle (2pt) node[below right=2pt, font=\scriptsize] {$(3,0)$};
  \fill (4,2) circle (2pt) node[right=2pt, font=\scriptsize] {$(4,2)$};

\end{tikzpicture}
\caption{The support of $f(x,y)$ and its Newton polygon, with lattice points labeled by coordinates.}
\end{figure}

The extreme monomials of $f$ are
\[
x^4y^2, \quad x^3, \quad xy^3, \quad y^2.
\]

\begin{definition}[Minkowski sum]
Let $P,Q \subset \mathbb{R}^2$. Their \emph{Minkowski sum} is
\[
P+Q := \{\, p+q \mid p\in P,\ q\in Q \,\}.
\]
If $P$ and $Q$ are convex polygons, then $P+Q$ is also a convex polygon.
\end{definition}

\begin{theorem}[Bernstein theorem, bivariate case] \cite{Minding}\cite{Bernstein}\cite{GKZ} \footnote{This bivariate case of Bernstein theorem was first proved by Minding \cite{Minding} in 1841. Bernstein generalized this result in 1975, see \cite{Bernstein}\cite[Theorem 2.8]{GKZ}.}
Let $f(x,y)$ and $g(x,y)$ be two nonzero polynomials with Newton polygons
$P = N(f)$ and $Q = N(g)$. Then the number of isolated common zeros in
$(\mathbb{C}^\ast)^2 := (\mathbb{C}\setminus\{0\})^2$, counted with multiplicity, satisfies
\[
\#\{(x,y)\in(\mathbb{C}^\ast)^2 : f(x,y)=g(x,y)=0\}
\;\le\;
\operatorname{Area}(P+Q) - \operatorname{Area}(P) - \operatorname{Area}(Q).
\]
Moreover, this bound is attained for generic coefficients.
\end{theorem}

Since each solution $(a,b)\in(\mathbb{C}^\ast)^2$ contributes a root to
$\operatorname{Res}_y(f,g) \in \mathbb{C}[x]$, we obtain:

\begin{corollary}
Let $f(x,y)$ and $g(x,y)$ be two bivariate polynomials with Newton polygons $P$ and $Q$. Then
\[
\deg_x\operatorname{Res}_y(f,g)\;\le\;
\operatorname{Area}(P+Q) - \operatorname{Area}(P) - \operatorname{Area}(Q),
\]
with equality for generic coefficients.
\end{corollary}

\noindent\textbf{Example.}
Let
\[
\begin{aligned}
f(x,y) &= 3x^2y + 2y^2 +2,\\
g(x,y) &= x^3 +  4x^2y + y^3 + 1.
\end{aligned}
\]
Then
\[
P=N(f)=\operatorname{Conv}\{(0,0),(0,2),(2,1)\},\qquad
Q=N(g)=\operatorname{Conv}\{(0,0),(0,3),(3,0)\}.
\]
Their Minkowski sum has vertices
\[
P+Q = \operatorname{Conv}\{(0,0),(5,1),(3,0),(0,5),(2,4)\}.
\]
Computing the areas gives
\[
\operatorname{Area}(P+Q)=31/2,\qquad \operatorname{Area}(P)=2,\qquad \operatorname{Area}(Q)=9/2,
\]
so
\[
\deg_x\operatorname{Res}_y(f,g)=\frac{31}{2} - 2 - \frac{9}{2} = 9.
\]
A direct computation shows that
\[
\operatorname{Res}_y(f,g) = -27x^9 - 48x^7 + 53x^6 + 36x^5 + 80x^4 + 16x^3 - 28x^2 + 16,
\]
which confirms that $\deg_x\operatorname{Res}_y(f,g) = 9$.
\section{Solutions for \texorpdfstring{$k$}{k} large}
By Theorem~\ref{theorem 2}, there are no solutions for $8 \le k \le 20$. 
This suggests that no solutions exist for any $k \ge 8$. 
In this section, rather than proving this full statement, we establish a weaker result: 
the system of equations associated with $f(q) = E_{2k}(q)$ has no solutions when $k$ is sufficiently large.
We follow an elimination strategy analogous to the case $k=2$, but now work over a symbolic coefficient ring so that the procedure can be performed uniformly for all $k$.

The coefficients in the polynomial system can be expressed as polynomials in the quantities $2^{2k-1},3^{2k-1},5^{2k-1},\dots$ and in the factor $B_{2k}/4k$. Since the number of equations eventually exceeds the number of variables, the elimination process produces a nontrivial polynomial relation among these parameters, which must be nonzero and hence implies that the system has no solutions.

To make this precise, we introduce the coefficient ring
\[
R = \mathbb{C}[x_0, x_2, x_3, x_5, \dots],
\]
where $x_0$ and each $x_p$ (with $p$ prime) are algebraically independent variables. Define a formal power series
\[
E(q) = x_0 +\sum_{n=1}^\infty y_n q^n \in R[[q]],
\]
where the coefficients $y_n$ are defined multiplicatively by
\begin{align*}
y_1 &= 1,\\
y_{p^n} &= \sum_{i=0}^n x_p^i \qquad (p\ \text{prime}),\\
y_{mn} &= y_m y_n \qquad ((m,n)=1).
\end{align*}

We now define a specialization map for each positive integer $k$ by
\[
\phi_k : R[[q]] \longrightarrow \mathbb{C}[[q]], \qquad
x_0 \mapsto -\frac{B_{2k}}{4k}, \quad x_p \mapsto p^{2k-1}.
\]
By construction, $\phi_k(y_n)=\sigma_{2k-1}(n)$, so $\phi_k(E(q))$ is a nonzero scalar multiple of $E_{2k}(q)$.

Suppose we can find some multiplicative $g(q)\in R[[q]]$ such that $E(q)g(q)$ is also multiplicative. Then the coefficients of the specialization $\phi_k(g(q))$ give a candidate solution for which $E_{2k}(q)\phi_k(g(q))$ has multiplicative coefficients. Therefore, to rule out solutions for large $k$, we can first show that the symbolic system over $R$ has no solution.

We now apply the same elimination framework as in Section~2, but carried out over $R$. Define
\[
E_{mn} = S_{mn} - \big(b_m S_n + b_n S_m - x_0 S_m S_n\big),
\quad\text{for } (m,n)=1,
\]
where
\[
S_n = \sum_{i=1}^{n-1} y_i\, b_{n-i}.
\]
We consider the system
\[
(\{b_n\}_{n\le 40};\; \{E_n\}_{n\le 40} \cup \{\,b_{mn}-b_m b_n\,\}_{mn\le 40}),
\]
and perform elimination exactly as in the case $k=2$.

Before proceeding, we isolate a key fact about specializations.

\begin{lem}\label{lemma3}
Let $f\in R$ be a non-constant polynomial in the variables $x_0,x_2,x_3,\dots$. Then $\phi_k(f)\neq 0$ for all sufficiently large $k$.
\end{lem}

\begin{proof}
Without loss of generality, assume that $f$ is irreducible.

If $x_0$ appears nontrivially in $f$, then $\phi_k(x_0)=-\frac{B_{2k}}{4k}$ grows faster than any fixed exponential as $k\to\infty$. The highest power of $\phi_k(x_0)$ in $f$ therefore dominates all lower powers of $\phi_k(x_0)$ and the contributions of the other variables, so $\phi_k(f)\neq 0$ for large $k$.

Suppose now that $f$ is independent of $x_0$. Then every monomial of $f$ is a product of powers of the variables $x_p$ ($p$ prime). Under specialization, we have $\phi_k(x_p)=p^{2k-1}$, and distinct monomials correspond to distinct $(2k-1)$-st powers of integers $n^{2k-1}$. The term with the largest $n$ dominates the others as $k\to\infty$, and hence the sum cannot vanish for all large $k$.
\end{proof}

As before, we eliminate all variables $b_n$ with $n \notin \mathbb{P}$ using the relations $b_{mn}=b_m b_n$. After these substitutions, only the variables $b_n$ with $n\in\mathbb{P}$ remain.

Next, by Lemma~\ref{lemma1}, for every $p\in\mathbb{P}$ with $p+1\notin\mathbb{P}$, the coefficient $b_p$ appears linearly in $E_{p+1}$, with
\[
\operatorname{coef}_{b_p}(E_{p+1}) = y_{p+1-p} = y_1 = 1,
\]
which remains nonzero under any specialization $\phi_k$. Therefore, each such $b_p$ may be eliminated by solving $E_{p+1}=0$ and substituting the result into the remaining system. After carrying out these eliminations, the reduced system becomes
\[
(\{b_2,b_3,b_4,b_7,b_8,b_{16},b_{31}\};\; \{E_{15},E_{21},E_{22},E_{34},E_{35},E_{36},E_{39},E_{40}\}).
\]

We now eliminate $b_8$, $b_{16}$, and $b_{31}$. Before doing so, we introduce the auxiliary polynomials
\[
\begin{aligned}
A &= -x_2^4 + 2x_2^3x_3 - x_2^2x_3^2 + x_2^2x_5 - x_2^2 - 2x_2x_3^2 - 4x_2x_3 + 2x_2x_5 - 2x_2 \\
  &\qquad {} - x_3^2 - 2x_3 + x_5 + x_7,\\[4pt]
B &= -2x_2^3 + x_2^2x_3 - 3x_2^2 + 2x_2x_3 - 2x_2 + x_3 + x_5,\\[4pt]
C &= -x_2^2 - 2x_2 + x_3.
\end{aligned}
\]
A computation in \texttt{SageMath} shows that
\[
\deg_{b_8}(E_{15}) = 1,
\qquad
\operatorname{coef}_{b_8}(E_{15}) = A\,x_0^2;
\]
\[
\deg_{b_{16}}(E_{21}) = 1,
\qquad
\operatorname{coef}_{b_{16}}(E_{21}) = A B\, x_0^3;
\]
\[
\deg_{b_{31}}(E_{34}) = 1,
\qquad
\operatorname{coef}_{b_{31}}(E_{34}) = A B C\, x_0^4.
\]

Hence we are able to eliminate $b_8$, $b_{16}$, and $b_{31}$. After these eliminations, the system becomes
\[
(\{b_2,b_3,b_4,b_7\};\; \{E_{22},E_{35},E_{36},E_{39},E_{40}\}).
\]
From this point onward, we replace each $E_n$ by its numerator, since $E_n=0$ is equivalent to its numerator being zero, and this simplifies the computations in \texttt{SageMath}.

We now begin the resultant phase. Define
\[
F_{35} = \operatorname{Res}_{b_7}(E_{22},E_{35}),\qquad
F_{36} = \operatorname{Res}_{b_7}(E_{22},E_{36}),\qquad
F_{39} = \operatorname{Res}_{b_7}(E_{22},E_{39}),\qquad
F_{40} = \operatorname{Res}_{b_7}(E_{22},E_{40}),
\]
and then
\[
G_{36} = \operatorname{Res}_{b_4}(F_{35},F_{36}),\qquad
G_{39} = \operatorname{Res}_{b_4}(F_{35},F_{39}),\qquad
G_{40} = \operatorname{Res}_{b_4}(F_{35},F_{40}).
\]
Here $b_7$ is linear in $E_{22}$, and we write $E_{22}=P_{22}b_7+Q_{22}$ with $P_{22},Q_{22}$ independent of $b_7$. As in the $k=2$ case, $(P_{22})^2$ divides each $G_n$ for $n\in\{36,39,40\}$. We then set
\[
H_{36} = \operatorname{Res}_{b_3}\!\Big(\tfrac{G_{39}}{(P_{22})^2},\;\tfrac{G_{36}}{(P_{22})^2}\Big),\qquad
H_{40} = \operatorname{Res}_{b_3}\!\Big(\tfrac{G_{39}}{(P_{22})^2},\;\tfrac{G_{40}}{(P_{22})^2}\Big).
\]
This leads to two subsystems:
\[
\text{sys}_{1} : \big(\{b_2\};\ \{H_{36},H_{40}\}\big),\qquad
\text{sys}_{2} : \big(\{b_2,b_3,b_4,b_7\};\ \{P_{22},Q_{22},E_{35},E_{36},E_{39},E_{40}\}\big).
\]

For $\text{sys}_1$, define
\[
I_{40} = \operatorname{Res}_{b_2}(H_{36},H_{40}).
\]
If $I_{40}\neq 0$, then $\text{sys}_1$ has no solution in $R$.

For $\text{sys}_2$, set
\[
R_{22} = \operatorname{Res}_{b_3}(P_{22},Q_{22}),\qquad
R_{35} = \operatorname{Res}_{b_3}(P_{22},E_{35}),\qquad
R_{36} = \operatorname{Res}_{b_3}(P_{22},E_{36}),\qquad
R_{39} = \operatorname{Res}_{b_3}(P_{22},E_{39}),
\]
then
\[
T_{36} = \operatorname{Res}_{b_7}(R_{35},R_{36}),\qquad
T_{39} = \operatorname{Res}_{b_7}(R_{35},R_{39}),
\]
and
\[
U_{36} = \operatorname{Res}_{b_4}(R_{22},T_{36}),\qquad
U_{39} = \operatorname{Res}_{b_4}(R_{22},T_{39}).
\]
Finally, set
\[
V_{39} = \operatorname{Res}_{b_2}(U_{36},U_{39}).
\]
If $V_{39}\neq 0$, then $\text{sys}_2$ has no solution in $R$.

By Lemma~\ref{lemma3}, when $k$ is sufficiently large, taking resultants commutes with specialization $\phi_k$ (because the degrees in each eliminated variable agree before and after specialization). Therefore $\phi_k(H_{n}) = H_{k,n}$ for $n=36,40$ and $\phi_k(U_n) = U_{k,n}$ for $n=36,39$. The remaining task is to prove that $I_{40}$ and $V_{39}$ are nonzero. Directly computing $I_{40}$ and $V_{39}$ in \texttt{SageMath} is infeasible because of size, so we proceed differently:

\begin{itemize}[leftmargin=2em]
\item First, we use Newton polygons (and the Bernstein bound) to predict the \emph{generic} degrees of $H_n$ and $U_n$.
\item Second, we evaluate at a small test value of $k$ for which the specialized versions $H_{k,n}$ and $U_{k,n}$ are explicitly computable.
\end{itemize}

We want:
\begin{itemize}[leftmargin=2em]
\item  $\deg H_{k,36}, \deg H_{k,40}$ match the predicted generic degree of $H_{36}, H_{40}$;
\item $\deg U_{k,36}, \deg U_{k,39}$  match the predicted generic degree of $U_{36},U_{39}$;
\item and also that $\gcd(H_{k,36},H_{k,40})=1$ and $\gcd(U_{k,36},U_{k,39})=1$.
\end{itemize}
If all of this holds for some test value of $k$, then $I_{40}$ and $V_{39}$ are genuinely nonzero in $R$. Lemma~\ref{lemma3} then implies that, for $k$ sufficiently large, neither $\text{sys}_1$ nor $\text{sys}_2$ has a solution.

From this point onward, we record the extreme monomials of each resultant and track the Newton polygons.

For $\text{sys}_1$, the triples $(X,Y,Z)$ such that $b_2^X b_3^Y b_4^Z$ is an extreme monomial of $F_{35}$ are
\[
\begin{aligned}
&(7,0,0), (5,1,0), (5,0,1), (4,2,0), (4,0,1), (3,1,1), (3,0,2), (2,3,0), (2,2,1),\\
&(2,0,2), (0,3,0), (1,2,1), (1,1,2), (1,0,3), (0,4,0), (0,3,1), (0,0,3).
\end{aligned}
\]
For $F_{36},F_{39},F_{40}$, the extreme triples are
\[
\begin{aligned}
&(7,0,0), (5,1,0), (5,0,1), (4,2,0), (4,0,1), (3,1,1), (3,0,2), (2,3,0), (2,2,1),\\
&(2,0,2), (1,3,0), (1,2,1), (1,1,2), (1,0,3), (0,4,0), (0,3,1), (0,2,2), (0,0,3).
\end{aligned}
\]

To control the degree of the $b_4$–resultants abstractly, set dummy polynomials of the same extreme support:
\[
\begin{aligned}
\text{check}_{F_{35}} ={}& a z^3 + b x^3 z^2 + c x y z^2 + d x^5 z + e x^3 y z + f x^2 y^2 z \\
        & {}+ g y^3 z + h x^7 + i x^5 y + j x^4 y^2 + k x^2 y^3 + o y^4,\\
\text{check}_{F_{36,39,40}} ={}& a' x z^3 + b' x^3 z^2 + c' x y z^2 + p' y^2 z^2 + d' x^5 z + e' x^3 y z + f' x^2 y^2 z \\
        & {}+ g' y^3 z + h' x^7 + i' x^5 y + j' x^4 y^2 + k' x^2 y^3 + o' y^4,
\end{aligned}
\]
and consider $\operatorname{Res}_z(\text{check}_{F_{35}},\text{check}_{F_{36,39,40}})$. A direct computation shows that the pairs $(X,Y)$ for which $x^X y^Y$ is an extreme monomial of this resultant are
\[
\begin{aligned}
&(24,0),(22,1),(21,2),(19,3),(18,4),\\
&(16,5),(15,6),(13,7),(12,8),(10,9),\\
&(8,10),(6,11),(4,12),(2,13),(0,14).
\end{aligned}
\]
First, we note that
\[
\operatorname{coef}_{x^{11}y^8}\!\big(\operatorname{Res}_z(\text{check}_{F_{35}},\text{check}_{F_{36,39,40}})\big)\neq 0.
\]
Second, the coefficient of $x^{12}y^8$ equals
\[
a'\cdot(fa' - bp')\cdot(j^2 p' - f j f' + f^2 j').
\]
In our application, the factor $fa'-bp'$ corresponds to
\[
\operatorname{coef}_{b_2^2 b_3^2 b_4}(F_{35})\cdot \operatorname{coef}_{b_2 b_4^3}(F_{n})
\;-\;
\operatorname{coef}_{b_2^3 b_4^2}(F_{35})\cdot \operatorname{coef}_{b_3^2 b_4^2}(F_{n}),
\qquad n\in\{36,39,40\}.
\]
A \texttt{SageMath} computation shows that this combination vanishes for each $n\in\{36,39,40\}$; hence the monomial $b_2^{12} b_3^8$ does not occur in $G_n$ for $n\in\{36,39,40\}$. On the other hand, computing $G_{2,n}$ for $n\in\{36,39,40\}$ directly shows that all other extreme monomials match, and that $b_2^{11}b_3^8$ does occur in $G_n$. Therefore, generically, the extreme pairs $(X,Y)$ for the monomials $b_2^X b_3^Y$ occurring in $G_n$ are
\[
\begin{aligned}
&(24,0),(22,1),(21,2),(19,3),(18,4),\\
&(16,5),(15,6),(13,7),(11,8),(10,9),\\
&(8,10),(6,11),(4,12),(2,13),(0,14).
\end{aligned}
\]

The extreme monomials of $P_{22}$ are $b_2^2$ and $b_3$. Consequently, the extreme pairs $(X,Y)$ for $G_n/(P_{22})^2$ are
\[
\begin{aligned}
&(20,0),(18,1),(17,2),(15,3),(14,4),\\
&(12,5),(11,6),(9,7),(7,8),(6,9),\\
&(4,10),(2,11),(0,12).
\end{aligned}
\]
Let $N(G_n/(P_{22})^2)$ denote the Newton polygon of $G_n/(P_{22})^2$ in the $(X,Y)$–plane. Then the vertices of $N(G_n/(P_{22})^2)$ are
\[
(20,0),\ (11,6),\ (6,9),\ (0,12),\ (0,0),
\]
and the vertices of $N(G_n/(P_{22})^2)+N(G_n/(P_{22})^2)$ are
\[
(40,0),\ (22,12),\ (12,18),\ (0,24),\ (0,0).
\]
We compute
\[
\operatorname{Area}\big(N(G_n/(P_{22})^2)\big)=\frac{255}{2},\qquad
\operatorname{Area}\big(N(G_n/(P_{22})^2)+N(G_n/(P_{22})^2)\big)=510.
\]
It follows from the Bernstein bound that, generically,
\[
\deg(H_{n})=510-\frac{255}{2}-\frac{255}{2}=255 \qquad (n\in\{36,40\}).
\]

Now pick a test value $k=6$, for which a direct computation is feasible. We find
\[
\deg_{b_2}(H_{6,36})=\deg_{b_2}(H_{6,40})=255.
\]
Thus $H_{36}$ and $H_{40}$ attain the generic degree $255$. Moreover, for $k=6$ we have
\[
\gcd(H_{6,36},H_{6,40})=1.
\]
Therefore
\[
I_{40}=\operatorname{Res}_{b_2}(H_{36},H_{40})
\]
is a nonzero polynomial in $R$. By Lemma~\ref{lemma3}, $\text{sys}_{k,1}$ has no solutions when $k$ is large.

\medskip

For $\text{sys}_2$, it is possible to compute $R_n$ directly. The polynomial $R_{22}$ depends only on $b_2$ and $b_4$, and the extreme pairs $(X,Y)$ for the monomials $b_2^Xb_4^Y$ occurring in $R_{22}$ are
\[
(5,0),\ (3,1),\ (0,2).
\]
The polynomials $R_{35}, R_{36}, R_{39}$ lie in $R[b_2,b_4,b_7]$. The extreme triples $(X,Y,Z)$ for $b_2^Xb_4^Yb_7^Z$ occurring in $R_{35}$ are
\[
(6,0,0),\ (4,1,0),\ (3,0,1),\ (0,2,0),\ (0,1,1),
\]
and in $R_{36},R_{39}$ are
\[
(6,0,0),\ (4,1,0),\ (3,0,1),\ (2,2,0),\ (1,1,1).
\]
Introduce dummy polynomials
\[
\begin{aligned}
\text{check}_{R_{35}} ={}& a x^6 + b x^4y + c x^3z + d y^2 + e yz + f,\\
\text{check}_{R_{36,39}} ={}& a' x^6 + b' x^4y + c' x^3z + d' y^2 z^2 + e' xyz + f'.
\end{aligned}
\]
Computing $\operatorname{Res}_{z}(\text{check}_{R_{35}}, \text{check}_{R_{36,39}})$, we find that the extreme pairs $(X,Y)$ for $x^Xy^Y$ in $\text{check}_{T_{36,39}}$ are
\[
(9,0),\ (7,1),\ (5,2),\ (2,3).
\]
Here the term $b_2^5b_4^2$ is special: we obtain
\[
\operatorname{coef}_{x^5y^2}\!\big(\operatorname{Res}_{z}(\text{check}_{R_{35}}, \text{check}_{R_{36}})\big) = cd'-be',
\]
which corresponds in $T_{36}$ to
\[
\operatorname{coef}_{b_2^3b_7}(R_{35})\cdot\operatorname{coef}_{b_4^2b_7^2}(R_{36})
\;-\;
\operatorname{coef}_{b_2^4b_4}(R_{35})\cdot\operatorname{coef}_{b_2b_4b_7}(R_{36}),
\]
and in $T_{39}$ to
\[
\operatorname{coef}_{b_2^3b_7}(R_{35})\cdot\operatorname{coef}_{b_4^2b_7^2}(R_{39})
\;-\;
\operatorname{coef}_{b_2^4b_4}(R_{35})\cdot\operatorname{coef}_{b_2b_4b_7}(R_{39}).
\]
Using \texttt{SageMath}, both combinations are zero. On the other hand, specializing $k=2$ confirms that the term $b_2^4b_4^2$ is nonzero in $T_{2,36}$ and $T_{2,39}$. Therefore, the extreme pairs $(X,Y)$ for $b_2^Xb_4^Y$ in $T_{36},T_{39}$ are
\[
(9,0),\ (7,1),\ (4,2),\ (2,3).
\]

Next, set
\[
\begin{aligned}
\text{check}_{R_{22}} ={}& a x^5 + b x^3y + c y^2 + d,\\
\text{check}_{T_{36,39}} ={}& a' x^9 + b' x^7y + c' x^4 + d' x^2y^3 + e'.
\end{aligned}
\]
Computing $\operatorname{Res}_{y}(\text{check}_{R_{22}}, \text{check}_{T_{36,39}})$, we find that its degree in $x$ is $20$. For this subsystem we take $k=2$ as our test value. We obtain
\[
\deg(U_{2,36})=\deg(U_{2,39})=20
\quad\text{and}\quad
\gcd(U_{2,36}, U_{2,39})=1.
\]
Therefore,
\[
V_{39} = \operatorname{Res}_{b_2}(U_{36},U_{39})
\]
is a nonzero polynomial in $R$. By Lemma~\ref{lemma3}, $\text{sys}_{k,2}$ has no solutions when $k$ is large.

\end{document}